\documentclass[10pt, leqno]{amsart}

\usepackage[T1]{fontenc}	
\usepackage{graphicx}
\usepackage{amsmath, amsthm, amssymb,centernot}

\usepackage{upgreek}
\usepackage{mathrsfs}

\usepackage{pdfsync}

\usepackage[usenames,dvipsnames]{xcolor}
\usepackage[inline,shortlabels]{enumitem}
\usepackage[sort,numbers]{natbib}								
\usepackage[hyphens]{url}									
\usepackage{verbatim}								

\usepackage{dsfont}									

\usepackage[pdftex,bookmarks,
						colorlinks,				
						breaklinks,unicode,
						citecolor=OliveGreen,
						linkcolor=Maroon
						]{hyperref} 

\usepackage{mathtools}
\usepackage{caption}
\usepackage{subcaption}

\usepackage{tikz}
\usetikzlibrary{calc, positioning, shapes.geometric,patterns}

\newlength{\hatchspread}
\newlength{\hatchthickness}
\newlength{\hatchshift}
\newcommand{\hatchcolor}{}
\tikzset{hatchspread/.code={\setlength{\hatchspread}{#1}},
         hatchthickness/.code={\setlength{\hatchthickness}{#1}},
         hatchshift/.code={\setlength{\hatchshift}{#1}},
         hatchcolor/.code={\renewcommand{\hatchcolor}{#1}}}
\tikzset{hatchspread=3pt,
         hatchthickness=0.4pt,
         hatchshift=0pt,
         hatchcolor=black}
\pgfdeclarepatternformonly[\hatchspread,\hatchthickness,\hatchshift,\hatchcolor]
   {custom north west lines}
   {\pgfqpoint{\dimexpr-2\hatchthickness}{\dimexpr-2\hatchthickness}}
   {\pgfqpoint{\dimexpr\hatchspread+2\hatchthickness}{\dimexpr\hatchspread+2\hatchthickness}}
   {\pgfqpoint{\dimexpr\hatchspread}{\dimexpr\hatchspread}}
   {
    \pgfsetlinewidth{\hatchthickness}
    \pgfpathmoveto{\pgfqpoint{0pt}{\dimexpr\hatchspread+\hatchshift}}
    \pgfpathlineto{\pgfqpoint{\dimexpr\hatchspread+0.15pt+\hatchshift}{-0.15pt}}
    \ifdim \hatchshift > 0pt
      \pgfpathmoveto{\pgfqpoint{0pt}{\hatchshift}}
      \pgfpathlineto{\pgfqpoint{\dimexpr0.15pt+\hatchshift}{-0.15pt}}
    \fi
    \pgfsetstrokecolor{\hatchcolor}
    \pgfusepath{stroke}
   }

\pgfdeclarepatternformonly[\hatchspread,\hatchthickness,\hatchshift,\hatchcolor]
   {custom north east lines}
   {\pgfqpoint{\dimexpr-2\hatchthickness}{\dimexpr-2\hatchthickness}}
   {\pgfqpoint{\dimexpr\hatchspread+2\hatchthickness}{\dimexpr\hatchspread+2\hatchthickness}}
   {\pgfqpoint{\dimexpr\hatchspread}{\dimexpr\hatchspread}}
   {
    \pgfsetlinewidth{\hatchthickness}
    \pgfpathmoveto{\pgfqpoint{\dimexpr\hatchshift-0.15pt}{-0.15pt}}
    \pgfpathlineto{\pgfqpoint{\dimexpr\hatchspread+0.15pt}{\dimexpr\hatchspread-\hatchshift+0.15pt}}
    \ifdim \hatchshift > 0pt
      \pgfpathmoveto{\pgfqpoint{-0.15pt}{\dimexpr\hatchspread-\hatchshift-0.15pt}}
      \pgfpathlineto{\pgfqpoint{\dimexpr\hatchshift+0.15pt}{\dimexpr\hatchspread+0.15pt}}
    \fi
    \pgfsetstrokecolor{\hatchcolor}
    \pgfusepath{stroke}
   }

\usepackage{xparse}

\ExplSyntaxOn
\NewDocumentCommand{\makeabbrev}{mmm}
 {
  \yoruk_makeabbrev:nnn { #1 } { #2 } { #3 }
 }

\cs_new_protected:Npn \yoruk_makeabbrev:nnn #1 #2 #3
 {
  \clist_map_inline:nn { #3 }
   {
    \cs_new_protected:cpn { #2 } { #1 { ##1 } }
   }
 }
 \ExplSyntaxOff

\makeabbrev{\textbf}{tbf#1}{a,b,c,d,e,f,g,h,i,j,k,l,m,n,o,p,q,r,s,t,u,v,w,x,y,z,A,B,C,D,E,F,G,H,I,J,K,L,M,N,O,P,Q,R,S,T,U,V,W,X,Y,Z}

\makeabbrev{\textbf}{bf#1}{a,b,c,d,e,f,g,h,i,j,k,l,m,n,o,p,q,r,s,t,u,v,w,x,y,z,A,B,C,D,E,F,G,H,I,J,K,L,M,N,O,P,Q,R,S,T,U,V,W,X,Y,Z}

\makeabbrev{\textsf}{tsf#1}{a,b,c,d,e,f,g,h,i,j,k,l,m,n,o,p,q,r,s,t,u,v,w,x,y,z,A,B,C,D,E,F,G,H,I,J,K,L,M,N,O,P,Q,R,S,T,U,V,W,X,Y,Z}

\makeabbrev{\mathsf}{mss#1}{a,b,c,d,e,f,g,h,i,j,k,l,m,n,o,p,q,r,s,t,u,v,w,x,y,z,A,B,C,D,E,F,G,H,I,J,K,L,M,N,O,P,Q,R,S,T,U,V,W,X,Y,Z}

\makeabbrev{\mathfrak}{mf#1}{a,b,c,d,e,f,g,h,i,j,k,l,m,n,o,p,q,r,s,t,u,v,w,x,y,z,A,B,C,D,E,F,G,H,I,J,K,L,M,N,O,P,Q,R,S,T,U,V,W,X,Y,Z}

\makeabbrev{\mathrm}{mrm#1}{a,b,c,d,e,f,g,h,i,j,k,l,m,n,o,p,q,r,s,t,u,v,w,x,y,z,A,B,C,D,E,F,G,H,I,J,K,L,M,N,O,P,Q,R,S,T,U,V,W,X,Y,Z}

\makeabbrev{\mathbf}{mbf#1}{a,b,c,d,e,f,g,h,i,j,k,l,m,n,o,p,q,r,s,t,u,v,w,x,y,z,A,B,C,D,E,F,G,H,I,J,K,L,M,N,O,P,Q,R,S,T,U,V,W,X,Y,Z}

\makeabbrev{\mathcal}{mc#1}{A,B,C,D,E,F,G,H,I,J,K,L,M,N,O,P,Q,R,S,T,U,V,W,X,Y,Z}

\makeabbrev{\mathbb}{mbb#1}{A,B,C,D,E,F,G,H,I,J,K,L,M,N,O,P,Q,R,S,T,U,V,W,X,Y,Z}

\makeabbrev{\mathscr}{ms#1}{A,B,C,D,E,F,G,H,I,J,K,L,M,N,O,P,Q,R,S,T,U,V,W,X,Y,Z}

\makeabbrev{\mathrm}{#1}{
Id,id,ran,rk,diag,stab,ann,conv,pr,ev,tr,End,Hom,sgn,im,op,can,fin,ext,red,tot,
rot,usc,lsc,Lip,lip,bSymLip,osc,AC,loc,coz,z,erf,csch,arcsinh,area,
supp,Opt,Adm,Cpl,Geo,GeoOpt,GeoAdm,GeoCpl,reg,
bd,co,Ric,Exp,dExp,dist,seg,Seg,cut,fcut,Cut,SDiff,Iso,Isom,diam,cl,Homeo,Diff,Der,vol,dvol,inj,relint, Graph,sub,Tube,codim,thick,thin,core,length,cusps,
var,law,Var,Poi,Gam,pa,so,iso,fs,inv,pqi,mix,Cov,
TestF,
}

\makeabbrev{\mathsf}{#1}{CD,BE,MCP,Ent,wMTW,MTW,Ch,RCD,EVI,Rad,dRad,SL,cSL,dSL,ScL,Irr,SC,wFe,VA}

\makeabbrev{\mathsc}{msc#1}{g}

\newcommand{\eps}{\varepsilon}

\renewcommand{\complement}{\mathrm{c}}

\newcommand{\mathsc}[1]{\text{\textsc{#1}}}

\newcommand{\emparg}{{\,\cdot\,}}

\DeclareMathOperator{\eqdef}{\coloneqq}

\renewcommand\qedsymbol{$\blacksquare$}

\let\epsilon\varepsilon

\newcommand{\rar}{\rightarrow}

\newcommand{\diff}{\mathop{}\!\mathrm{d}}						

\newcommand{\abs}[1]{\left\lvert#1\right\rvert}						
\newcommand{\norm}[1]{\left\lVert#1\right\rVert}					

\newcommand{\set}[1]{\left\{#1\right\}}							
\newcommand{\tonde}[1]{\left(#1\right)}							
\newcommand{\ttonde}[1]{\big({#1}\big)}
\newcommand{\quadre}[1]{\left[#1\right]}							

\DeclareSymbolFont{symbolsC}{U}{pxsyc}{m}{n}
\SetSymbolFont{symbolsC}{bold}{U}{pxsyc}{bx}{n}
\DeclareFontSubstitution{U}{pxsyc}{m}{n}
\DeclareMathSymbol{\medcirc}{\mathbin}{symbolsC}{7}
\DeclareSymbolFont{symbolsZ}{OMS}{pxsy}{m}{n}
\SetSymbolFont{symbolsZ}{bold}{OMS}{pxsy}{bx}{n}
\DeclareFontSubstitution{OMS}{pxsy}{m}{n}

\DeclareMathOperator{\emp}{\varnothing} 

\newcommand{\N}{{\mathbb N}}
\newcommand{\R}{{\mathbb R}}

\newcommand{\restr}{\big\lvert}

\usepackage{booktabs}

\allowdisplaybreaks

\usetikzlibrary{shapes.misc}
\tikzset{cross/.style={cross out, draw=black, minimum size=2*(#1-\pgflinewidth), inner sep=0pt, outer sep=0pt},
cross/.default={4pt}}

\newcommand{\iref}[1]{\ref{#1}}

\newcommand{\comma}{\,\,\mathrm{,}\;\,}

\newcommand{\fstop}{\,\,\mathrm{.}}

\newcommand{\av}[1]{\left\langle#1\right\rangle}
\newcommand{\card}[1]{\abs{#1}}

\usepackage{scrextend}						

\newcommand{\blue}[1]{{\color{black}{#1}}}

\newcommand{\prid}{\mathrel{\ooalign{$\lneq$\cr\raise.22ex\hbox{$\lhd$}\cr}}}

\newcommand{\UBL}{L}

\let\temp\phi
\let\phi\varphi
\let\varphi\temp

\newcommand{\be}{\begin{equation*}}
 \newcommand{\ee}{\end{equation*}}

\numberwithin{equation}{section}
\theoremstyle{plain}
\newtheorem{thm}{Theorem}[section]
\newtheorem*{thm*}{Theorem}
\newtheorem*{mthm*}{Main Theorem}

\newtheorem{lem}[thm]{Lemma}
\newtheorem{cor}[thm]{Corollary}

\theoremstyle{definition}
\newtheorem*{defs*}{Definition}

\theoremstyle{remark}
\newtheorem*{ass*}{Assumption}

\usepackage{marginnote}

\newcommand{\enl}[2]{\left(#1\right)_{#2}}

\renewcommand{\paragraph}[1]{\medskip\emph{#1.}\quad}

\begin{document}
\title{Effective contraction of Skinning maps}

\author[T.~Cremaschi]{Tommaso Cremaschi}
\address{Department of Mathematics, University of Southern California, Kaprelian Hall\\ 3620 S.\ Vermont Ave., Los Angeles, CA 90089-2532}
\email{cremasch@usc.edu}
\thanks{T.C.~was partially supported by the National Science  Foundation  under  Grant  No.~DMS-1928930  while participating  in  a program hosted by the Mathematical Sciences Research Institute in Berkeley, California, during the Fall 2020 semester.
}

\author[L.~Dello Schiavo]{Lorenzo Dello Schiavo}
\address{Institute of Science and Technology Austria, Am Campus 1, 3400 Klosterneuburg, Austria}
\email{lorenzo.delloschiavo@ist.ac.at}
\thanks{
L.D.S.\ gratefully acknowledges funding of his current position by the Austrian Science Fund (FWF) grant F65, and by the European Research Council (ERC, grant No.~716117, awarded to Prof.\ Dr.~Jan Maas).
}

\begin{abstract}
Using elementary hyperbolic geometry, we give an explicit formula for the contraction constant of the skinning map over moduli spaces of relatively acylindrical hyperbolic manifolds.
\end{abstract}



\maketitle

\section{Introduction}

Let~$M_1$, $M_2$ be hyperbolic manifolds \blue{of finite-type, i.e. the interior of compact 3-manifolds,} with incompressible boundary, and homeomorphic geometrically finite ends~$E_1\subset M_1$ and~$E_2\subset M_2$.
From a topological point of view, since~$M_1$ and~$M_2$ are tame, \cite{AG2004,CG2006}, the surfaces~$S_i$ corresponding to the boundary of the ends~$E_i$ are naturally homeomorphic.
We can thus glue the two manifolds via an orientation-reversing homeomorphism~$\tau$, and obtain a new topological $3$-manifold~$M=M_1\cup_\tau M_2$.
Usually, one seeks sufficient conditions for~$M$ to admit a complete hyperbolic metric, which is relevant, for example, in the proof of geometrization for hyperbolic manifolds, \cite{Kap2001}. We call this the \emph{glueing problem} for~$M$.
The \emph{skinning map}, described below, was first introduced by W.~P.~Thurston, exactly to study this glueing problem,~\cite{Th1982}.

The moduli space~$GF(M,\mathcal P)$ of all hyperbolic metrics on~$M$ with geometrically finite ends and parabolic locus~$\mathcal P$ is parameterised by the Teichm\"uller space~$\mcT(\partial_0 M)$ with~$\partial_0 M$ the closure in~$\partial M$ of \blue{the complement~$\mathcal P^\complement$ of $\mathcal P$}, viz.\ $GF(M,\mathcal P)=\mathcal T(\partial_0 M)$.
For simplicity, let us here assume that~$\mathcal P$ only contains toroidal boundary components of~$M$.
Now, let~$N\in GF(M,\mathcal P)$ be a uniformization, and~$S\in\pi_0(\partial_0 M)$ be a (non-toroidal) boundary component.
The cover of~$N$ associated to~$\pi_1(S)$ is a quasi-Fuchsian manifold~$N_S$.
The manifold~$N_S$ has two ends,~$A$ and~$B$, of which~$A$ is isometric to the end of~$M$ corresponding to~$S$.
One defines the skinning map~$\sigma_M$ at~$N$ as the conformal structure of the new end~$B$.
As it turns out, the skinning map is an analytic map~$\sigma_M\colon \mathcal T(\partial_0 M)\rar \mathcal T(\overline{\partial_0 M})$, where the bar denotes opposite orientation. 
The glueing instruction determines an isometry~$\tau^*\colon \mathcal T(\partial_0 M)\rar \mathcal T(\overline{\partial_0 M})$, and any fixed point of $\tau^*\circ\sigma_M$ gives a solution to the glueing problem by the Maskit Combination Theorem, e.g.~\cite{Kap2001}.

Given a covering map between Riemann surfaces $\pi\colon Y\rar X$ the \emph{Poincar\'e series operator} is a push-forward operator $\Theta_{Y/X}:Q(Y)\rar Q(X)$, similar to the push-forward of measures, pushing quadrating differentials on~$Y$ to quadratic differentials on~$X$.

In \cite{McM1989}, C.~McMullen showed that the skinning map of an acylindrical manifold~$N$ is contracting, with contraction constant only depending on the topology of~$\partial_0 M$.
Furthermore, he related the skinning map to the Poincar\'e series operator~$\Theta$ by the following formula:
\begin{equation}\label{eq:Intro:McMullen}
\diff \sigma_M^*(\phi)=\sum_{U\in BN} \Theta_{U/X} \left(\phi\vert_U\right) \comma
\end{equation}
where $BN$ is \blue{a} collection of sub-surfaces of $\im(\sigma)$.
When~$M$ is acylindrical and~$P=\emp$, we have that $BN$ is just a collection of disks, the \emph{leopard spots} of~\cite{McM1989}.
If~$P\neq\emp$ and~$M$ is relatively acylindrical, then we can also have punctured disks coming from peripheral cylinders of $M$.

As a consequence of~\eqref{eq:Intro:McMullen}, one can \blue{estimate} the operator norm of the co-derivative map~$\diff\sigma_M^*$ of the skinning map by bounding the Poincar\'e series operator of the corresponding surfaces.
Using \blue{such estimate}, we provide here effective bounds, in terms of the topology of~$\partial_0 M$, on the contraction of the skinning map in the acylindrical case.
This builds on previous work~\cite{BarDil96} of D.~E.~Barret and J.~Diller, who gave an alternative proof of McMullen's estimates on the norm of the Poincar\'e operator, \cite{McM1989}.

Improving on the main result of~\cite{BarDil96} (Thm.~\ref{t:BD} below), we show:

\begin{thm}\label{t:Main} Suppose $X$ is a Riemann surface of finite-type and let $Y$ be a disk or a punctured disk.
Further let~$\pi\colon Y\rar X$ be a holomorphic covering map.
Then, the norm of the corresponding Poincar\'e series operator satisfies:
\begin{align*}
\norm{\Theta}_\op < \frac{1}{1+C_{g,n,\ell}} <1
\end{align*}
\blue{for some constant~$C_{g,n,\ell}>0$ depending only on the topology of $X\cong S_{g,n}$ and the injectivity radius $\ell$ of $X$.}
\end{thm}

In contrast with~\cite{BarDil96}, we compute the contraction constant~$C_{g,n,\ell}$ in a completely explicit way and in the case \blue{under examination} without any extra assumptions on~$\norm{\Theta}_\op $.
The constant~$C_{g,n,\ell}$ only depends on: the genus~$g$ of~$X$, the number of punctures~$n$ of~$X$, the length~$\ell$ of the shortest closed geodesic in~$X$.
So, we obtain an explicit bound over the moduli space of geometrically finite hyperbolic manifolds.

Furthermore,~$C_{g,n,\ell}$ is continuous and decreasing as a function of $\ell$, in fact it is linear in $\ell$, and satisfies the following asymptotic expansion for $g,n\gg 1$.
Let~$\chi\eqdef 2g-2+n$ be the Euler characteristic, and~$\kappa\eqdef 3g-3+n$ be the complexity of~$X$.
Then,
\begin{align*}
\log \log \tonde{\tfrac{\ell}{C_{g,n,\ell}} } \asymp \tfrac{4}{\arcsinh(1)}\,\chi^2+\coth\ttonde{\tfrac{\pi}{12}}\,\chi + \pi \sinh\tonde{\tfrac{1}{2}\arcsinh\ttonde{\tanh(\pi/12)}}\,\kappa\fstop
\end{align*}

\paragraph{An application to infinite-type 3-manifolds}
In~\cite{C2018b} the first named author studied the class~$\mathcal M^B$ of infinite-type 3-manifolds~$M$ admitting an exhaustion~$M=\cup_i M_i$ by hyperbolizable 3-manifolds~$M_i$ with incompressible boundary and with uniformly bounded genus.

One can use skinning maps to study the space of hyperbolic metrics on the manifolds in $\mathcal M^B$ that admit hyperbolic structures.
Indeed, consider all manifolds~$M\in \mathcal M^B$ such that for all $i\in\N$ every component $U_i\eqdef \overline{M_i\setminus M_{i-1}}$ is acylindrical. 
By the main results of \cite{C2018b} this guarantees that~$M$ is in fact hyperbolic, which is in general not the case, see~\cite{C20171,C2018c}, or \cite{CS2018,CVP2020} for other examples of infinite-type hyperbolic 3-manifolds.
We can thus think of a (hyperbolic) metric~$g$ on~$M$ as a gluing of (hyperbolic) metrics~$g_i$ on the~$U_i$'s and so it makes sense to investigate the glueing of pairs~$U_i$,~$U_{i+1}$ via skinning maps.

In order to approach the construction of~$g$ in this way, it is helpful to know that the contraction factor of the skinning maps over the Teichm\"uller spaces relative to~$U_i$ stays well below~$1$ uniformly in~$i$.
The latter fact follows from Theorem~\ref{t:Main}, in view of the uniform bound on the genus of the~$M_i$'s.

	 
\section{Notation}
Throughout the work,~$X$ is a hyperbolic Riemann surface of finite type.
Let~$\overline X$ be the compact Riemannian surface obtained by adding a single point to each end of~$X$.
We indicate by
\begin{itemize}
\item $g$ the genus of~$X$;
\item $n$ the cardinality of the set of punctures~$P\eqdef \overline{X}\setminus X$;
\end{itemize}
We may thus regard~$X$ as an element of the \blue{moduli space~$\mathcal M(S_{g,n})$} of the $n$-punctured Riemann surface of genus~$g$.
Further let
\begin{itemize}
\item $\chi\eqdef 2g-2+n$ be the Euler characteristic of~$X$;
\item $\kappa\eqdef 3g-3+n$ be the complexity of~$X$, with the exception of the surface~$S_{0,2}$ for which $\kappa\eqdef 0$.
\end{itemize}
We say that a curve in~$X$ is a \emph{short geodesic} if it is a closed geodesic of length less than $2\arcsinh(1)$, and we define
\begin{itemize}
\item $\Gamma$ the set of short geodesics on~$X$;
\item $\ell\eqdef \min_{\gamma \in \Gamma} \ell(\gamma)$ \blue{(twice)} the \emph{injectivity radius} of~$X$.
\end{itemize}

For any~$A\subset X$, denote by~$\card{A}$ the number of connected components of~$A$, and indicate by $U\in \pi_0(A)$ any of such connected components.
Let~$d$ be the intrinsic distance of~$X$ and further set
\begin{align*}
\enl{A}{s}\eqdef \set{x\in X : \dist(x,A)\leq s}\comma \qquad s>0\fstop
\end{align*}

\paragraph{Regions} Denote by~$D$ the Poincar\'e disk, and set~$D^*\eqdef D\setminus\set{0}$.
The \emph{cusp}~$\mcC_p$ about~$p\in P$ is the image of the punctured disk~$\set{0<\abs{z}<e^{-\pi}}$ under the holomorphic cover~$\pi_p\colon D^*\to X$ about~$p$.

We start by recalling the following well-known fact.
\begin{lem}[{\cite[Thm.~4.1.1]{Bu1992}}]\label{l:Buser}
Let~$\gamma$ be a short closed geodesic in~$X$ of length~$\ell(\gamma)$, and set~$w\eqdef \arcsinh\left(\frac 1{\sinh(\ell(\gamma)/2) }\right)$.
The collar~$\mcC_\gamma$ around~$\gamma$ is isometric to $[-w,w]\times \mbbS^1$ with the metric $\diff\rho^2+\ell(\gamma)^2\cosh^2(\rho)\diff t^2$.
\end{lem}
\blue{Note that in the previous statement the local metric, in Fermi coordinates, is parametrised with $\ell$ speed hence the $\ell^2$ factor.}

We define:
\begin{itemize}
\item the \emph{cusp part}~$X_\cusps$ of~$X$ as~$X_\cusps\eqdef\cup_{p\in P} \, \mcC_p$;
\item the \emph{core}~$X_\core$ of~$X$ as~$X_\core\eqdef X\setminus X_\cusps$;
\item the \emph{thick part}~$X_\thick$ of~$X$ as~$X_\thick\eqdef X_\core\setminus \cup_{\gamma\in \Gamma}\, \mcC_\gamma$;
\item the \emph{thin part}~$X_\thin$ of~$X$ as~$X_\thin\eqdef \overline{X\setminus X_\thick}$.
\end{itemize}

\paragraph{Quadratic differentials}
Let~$T_{1,0}^*X$ be the holomorphic cotangent bundle of~$X$.
A quadratic differential on~$X$ is any section~$\psi$ of~$T_{1,0}^*X\otimes T_{1,0}^*X$, satisfying, in local coordinates,~$\psi=\psi(z)\diff z^2$.
A quadratic differential~$\psi$ is \emph{holomorphic} if its local trivializations~$\psi(z)$ are holomorphic.
To each holomorphic quadratic differential~$\psi$ we can associate a measure~$\abs{\psi}$ on~$X$ defined by~$\abs{\psi}=\abs{\psi(z)}\cdot\abs{\diff z}^2$.
We denote by~$\av{\psi(\emparg)}$ the density of the measure~$\abs{\psi}$ with respect to the Riemannian volume of~$X$.

We say that any~$\psi$ as above is \emph{integrable} if~$\norm{\psi}\eqdef \abs{\psi}\!(X)$ is finite, and we denote by~$Q(X)$ the space of all integrable holomorphic quadratic differentials on~$X$, endowed with the norm~$\norm{\emparg}$.
When~$X$ has finite topological type,~$Q(X)$ is finite-dimensional, its dimension depending only on~$g$ and~$n$.

\paragraph{Constants} Everywhere in this work,~$r,s,t,w$ and~$\eps$ are free parameters.
We shall make use of the following universal constants:
\begin{itemize}
\item $\eps_0\eqdef \arcsinh(1)\approx 0.8813$ the two-dimensional Margulis constant;
\item $c_1\eqdef \coth(\pi/12)\approx 3.9065$;
\item $c_2\eqdef\arcsinh\ttonde{\tanh(\pi/12)}\approx 0.2532$;
\item $c_3\eqdef \frac{\pi \sinh\tonde{\tfrac{1}{2}\arcsinh\ttonde{\tanh(\pi/12)}}}{\arcsinh(\tanh(\pi/12))}\approx 1.5750$;
\item $c_4\eqdef \ttonde{1-\tanh^2(1/2)}^2\approx 0.6185$;
\item $c_5\eqdef 4\pi\ttonde{1+\sinh(1)}\approx 27.3343$;
\item $c_6\eqdef (e c_4)^{e^{2c_3+2}} \approx 76.5904$;
\item $c_7\eqdef \max _x x\cdot\arcsinh\ttonde{\csch(x/2)}\approx 1.5536$.
\end{itemize}

Finally, for simplicity of notation, we shall make use of the following auxiliary constants, also depending on~$X$:
\begin{itemize}
\item $a_1\eqdef 4\abs{\chi}^2/\epsilon +2\kappa \log c_1+2\, c_2\, c_3$;
\item $a_2\eqdef \log(e\, c_4)\, e^{a_1+2(1+c_3)}$.
\end{itemize}

\blue{We denote by~$a\wedge b$ the minimum between two quantities~$a,b\in\R$.}

\section{Outline}

We start by recalling the results of D.E.~Barret and J.~Diller \cite{BarDil96} that we make explicit using classic hyperbolic geometry. The main result of \cite{BarDil96} is:
\begin{thm}[{\cite[Thm.~1.1]{BarDil96}}]\label{t:BD}
Suppose $X$, $Y$ are Riemman surfaces of finite-type and let $\pi\colon Y\rar X$ be a holomorphic covering map. Then, the norm of the corresponding Poincar\'e operator satisfies:
\begin{equation*}
\norm{\Theta}_{\op}\eqdef \sup_{\substack{\phi\in Q(Y)\\\norm{\phi}=1}} \norm{\Theta \phi} <1-k<1 \fstop
\end{equation*}
Furthermore, $k>0$ may be taken to depend only on the topology of $X$, $Y$, and the length $\ell$ of the shortest closed geodesic on $X$. As a function of $\ell$, \blue{the number} $k$ may be taken to be continuous and increasing.
\end{thm}

In order to prove the above theorem, consider a \blue{unit-norm} quadratic differential $\phi\in Q(Y)$ such that $\Theta\phi\neq 0$. In~\cite{BarDil96}, the authors estimate
\begin{equation*}
1-\norm{\Theta\phi}
\end{equation*}
as follows.
Let $K\subset \overline X$ be any compact set containing the set~$Z$ of zeroes of $\Theta\phi$ and the punctures of $X$, viz.\ $Z\cup P\subset K$, and such that $\partial K$ is smooth.
Further let
\begin{equation}\label{eq:DefMr}
m(r)\eqdef \min_{p\in\partial \enl K r}\av{\Theta\phi} \fstop
\end{equation}
Then, for every $t>1$ and every $r_0>0$, \blue{\cite[Lem.~3.2]{BarDil96} proves the following estimate}
\begin{equation}\label{keyestimate}
1-\norm{\Theta\phi}\geq \int_0^{r_0}m(r)\quadre{t^{-1}\area(X\setminus \enl K r)-\length(\partial \enl K r)}\diff r \fstop
\end{equation}
In general the $t$ in the above estimate will depend on the geometry and topology of the covering surface $Y$. \blue{In the case at hand however,~$Y$ is either the Poincar\'e disk or a punctured disk, and by work of J.~Diller \cite{Dil95}, we can assume that $t=1$.} It is likely that the constants of Diller can be made explicit as well and so that one could have a version of Theorem \ref{t:BD} were the constants are explicit in the topology of $X$, $Y$ and their injectivity radii.

In the following sections, we give effective estimates for $m(r)$, $\area(X\setminus \enl K r)$, and $\length(\partial \enl K r)$.
In order to estimate $m(r)$ we will need the following result from \cite{BarDil96}.

\begin{thm}[{\cite[Thm.~4.4]{BarDil96}}]\label{BD4.4}
\blue{
Let~$\psi\in Q(X)$ with zero set~$Z$.
Suppose $W\subset X\setminus Z$ is a domain such that $\av{\psi(p)} \leq \UBL$ for all $p\in W$, and set $\rho(p)\eqdef \min\set{1, \dist (p,\partial W)}$.
}
Then, if $\gamma\subset W$ is a path connecting $p_1$ and $p_2$ we have:
\begin{equation*}
\frac{\av{\psi(p_1)} }{ \av{\psi(p_2)} }\geq \left( \frac{ \av{\psi(p_2)} }{c_4\UBL} \right) ^{-1+\exp\left(\int_\gamma \frac{\diff s}{\tanh (\rho/2)}\right)} \fstop
\end{equation*}

\end{thm}

\section{Effective Computations}\label{explicitcomps}

The following is an easy lemma bounding the diameter of components of $\enl{X_\thick}{\eps}$ or $\enl{X_\core}{\eps}$.

\begin{lem} 
Let~\blue{$X\in \mathcal M(S_{g,n})$}. Then,
\begin{enumerate}[$(i)$]
\item\label{i:l:BD4.2:1} any pair of points in the same connected component of~$\enl{X_\thick}{\eps}$ is joined by a path of length at most~$4\abs{\chi}/\eps$;
\item\label{i:l:BD4.2:2} any pair of points in~$\enl{X_\core}{\eps}$ is joined by a path~$\gamma$ in~$\enl{X_\core}{\eps}$ satisfying
\begin{align}
\ell(\gamma)\leq 4\abs{\chi}^2/\eps+ 2\kappa\, \arcsinh\ttonde{\csch(\ell/2)} \fstop
\end{align}
\end{enumerate}
\begin{proof}
\blue{Assertion \iref{i:l:BD4.2:1} is a} consequence of the Bounded Diameter Lemma \cite{Th1978}.

\iref{i:l:BD4.2:2} Using the fact that each component of $\enl{X_\thick}{\eps}$ contains an essential pair of pants and that the maximal number of pairwise disjoint short curves is $\kappa$ we have:

\paragraph{Claim} $\card{\enl{X_\thick}{\eps}}\leq \abs{\chi}$ and~$\card{\enl{X_\thin}{\eps}}\leq \kappa$.

By short-cutting in the region we obtain:

\paragraph{Claim} A length-minimizing~$\gamma$ enters each~$U\in\pi_0\ttonde{\enl{X_\core}{\eps}}$, resp.~$U \in \pi_0\ttonde{\enl{X_\thin}{\eps}}$ at most once.

\medskip

Let~$\gamma$ be length-minimizing. By~\iref{i:l:BD4.2:1} we have~$\length(\gamma \cap U)\leq 4\abs{\chi}/\eps$.
By the Collar Lemma \cite{Bu1992},
\begin{align*}
\length(\gamma\cap U)\leq \diam (U)\leq 2\,\arcsinh\ttonde{\csch(\ell/2)} \fstop
\end{align*}

The conclusion follows combining the previous estimates with the two claims.
\end{proof}
\end{lem}

\blue{The next Lemma is~\cite[Lem.~4.6]{BarDil96}.
We just work out the constant explicitly.
}
\begin{lem} \label{4.6}
Let~$\UBL(s)\eqdef\displaystyle\max_{p\in \enl{X_\thick}{s}} \av{\psi(p)}$. Then,
\begin{enumerate}[$(i)$]
\item\label{i:l:BD4.6:1} $\UBL(0)\geq \tfrac{\ell \wedge 1}{16\, \abs{\chi}}\norm{\psi}$;
\item\label{i:l:BD4.6:2} for all~$0\leq s\leq t$, we have~$\UBL(s)\geq e^{s-t}\UBL(t)$.
\end{enumerate}

\begin{proof}
\iref{i:l:BD4.6:1} Firstly assume that at most half the mass of $\psi$ is concentrated inside the collars of short geodesics. As in~\cite[Lem.~4.6(i)]{BarDil96}, it follows that
\begin{equation}\label{eq:l:BD4.6:1}
\av{\psi}\geq \frac{\norm{\psi}}{2\, \area(X)}=\frac{\norm{\psi}}{4\pi\abs{\chi}} \geq \frac{\norm{\psi}}{16\,\abs{\chi}} \fstop
\end{equation}

Assume now that at least half the mass of $\psi$ is concentrated inside collars of short geodesics.
Let~$\gamma$ be any such geodesic and let \blue{$\mcC\eqdef\mcC_\gamma$ be the collar around $\gamma$}.
For~$r\leq R\eqdef \pi^2/\ell(\gamma)$ and $r$ satisfying~$\tan\ttonde{\pi r/(2R)}=\csch\ttonde{\ell(\gamma)/2}$, we have that
\begin{align*}
\frac{1}{2\, \area(X)}\norm{\psi}\leq& \int_\mcC \abs{\psi} = \int_0^{2\pi} \int_{e^{-r}}^{e^r} \frac{\abs{f(z)}}{\abs{z}^2}r \diff r\diff\theta
\\
\leq& \int_0^{2\pi}\int_{e^{-r}}^{e^r} \UBL r^{-1} \diff r\diff\theta =4\pi \UBL r \comma
\intertext{hence that}
\frac{\norm{\psi}}{2\pi r\, \area(X) } \leq&\ 4 \UBL \fstop
\end{align*}

Computing both~$r$ and~$R$ in terms of~$\ell(\gamma)$,
\begin{align*}
\UBL(0)\geq&\ \max_{\partial\mcC} \av{\psi} = \frac{4\UBL R^2}{\pi^2}\cos^2 \frac{\pi r}{2R }
\\
\geq&\ \frac{\norm{\psi}}{2\, \area(X)} \frac{R^2}{2R\arctan\tonde{\csch\ttonde{\ell(\gamma)/2}}} \cos^2\tonde{\arctan\tonde{\csch\ttonde{\ell(\gamma)/2}}} \fstop
\intertext{Now, since~$\cos^2\tonde{\arctan\tonde{\csch(t)}}=\tanh^2(t)$, and substituting~$R\eqdef 2\pi/\ell(\gamma)$,}
\UBL(0)\geq &\ \frac{\norm{\psi}}{4\, \area(X)} \frac{R\, \tanh^2\ttonde{\ell(\gamma)/2}}{\arctan\tonde{\csch\ttonde{\ell(\gamma)/2}}}
\\
=&\ \frac{\pi^2\norm{\psi}}{4\, \area(X)} \frac{\tanh^2\ttonde{\ell(\gamma)/2}}{\ell(\gamma)^2\cdot\arctan\ttonde{\csch(\ell(\gamma)/2)}} \cdot \ell(\gamma) \fstop
\intertext{Since~$t\mapsto \tanh^2(t/2)/\ttonde{t^2\arctan(\csch(t/2))}$ has global minimum~$\tfrac{1}{2\pi}$ at~$t=0$, we have that}
\UBL(0)\geq&\ \frac{\pi\, \ell(\gamma)}{8\, \area(X)} \norm{\psi} \geq \frac{\ell}{16\abs{\chi}} \norm{\psi}\fstop
\end{align*}

Combining the above inequality with~\eqref{eq:l:BD4.6:1} yields the assertion.

\iref{i:l:BD4.6:2} is~\cite[Lem.~4.6]{BarDil96}.
\end{proof}
\end{lem}

Let $\log_+(x)\eqdef \max\set{0,\log(x)}$.
We start with some estimates towards establishing \eqref{keyestimate}.

\begin{lem}\label{l:BD5.1}
For each connected component~$U\in \pi_0\ttonde{\enl{X_\thick}{s}}$, letting~$s=\log_+(c_1 t)$
\begin{enumerate}[$(i)$]
\item\label{i:l:BD5.1:1} $\area(U) - t\, \length(\partial U)\geq \pi/3$;
\item\label{i:l:BD5.1:2} for all $p\in U$: $\inj_p\geq c_2/t$;
\item\label{i:l:BD5.1:3} given $p_1,p_2\in U$ there exists~$\gamma\subset U$ connecting~$p_1$ and~$p_2$ such that
\begin{equation*}
\ell(\gamma)\leq \frac{4\abs{\chi}^2} \epsilon +2\kappa \log t +2\kappa \log c_1\fstop
\end{equation*}
\end{enumerate}

\begin{proof}
\iref{i:l:BD5.1:1} Let~$g_U$ and~$n_U$ respectively denote the genus of $U$ and the number of boundary components of~$U$.
Further let~$A_1,\dotsc, A_{n_U}$ denote the embedded annuli bounded by short closed geodesics on one side and by connected components of~$\partial U$ on the other side. We allow for~$A_j$ being part of a cusp, in which case, on one side, it is bounded by a puncture rather than by a short geodesic.

By the Gauss--Bonnet Theorem,
\begin{align*}
\area(U)=2\pi(2g_U + n_U-2)-\sum \area(A_j) \fstop
\end{align*}

If~$n_U=0$ then $U=X$, which yields~$\area(U)-t\, \length(\partial U)=2\pi\abs{\chi}$.
Thus, in the following we may assume without loss of generality that~$n_U\geq 1$.
In this case, either~$g_U\geq 1$ and $n_U\geq 1$, or~$g_U=0$ and~$n_U\geq 3$. Thus,
\begin{align*}
\area(U)\geq 2\pi\, \frac{n_U}{3} -\sum_j\area(A_j) \fstop
\end{align*} 

Let~$\ell_j$ denote the length of the geodesic component of~$\partial A_j$ and $L_j$ denote the length of the other component. Then,
\begin{align*}
\area(U)-t\, \length(\partial U) \geq 2\pi\, \frac{n_U}{3} + \sum_j\ttonde{(t-1)\area(A_j)-t(\area(A_j)+L_j)}\fstop
\end{align*}

By Lemma~\ref{l:Buser}, setting
\begin{equation}\label{eq:wAnnuli}
w_j\eqdef \arcsinh\left(\frac 1{\sinh(\ell_j/2) }\right)\comma
\end{equation}
we have that
\begin{equation*}
\area(A_j)=\int_0^{w_j-s} \int_0^1\ell_j\cosh(\rho)\diff\rho\diff t=\ell_j\sinh(w_j-s)
\end{equation*}
and
\begin{equation*}
L_j=\ell_j\cosh(w_j-s)\fstop
\end{equation*}
We see that 
\begin{align*}
\area(A_j)+L_j=\ell_j \ttonde{\sinh(w_j-s)+\cosh(w_j)}=\frac{e^{-s}\ell_j}{\tanh(\ell_j/4)}
\end{align*}
is monotone increasing in~$\ell_j$ (e.g.\ by differentiating w.r.t.\ $\ell_j$). Thus it achieves its minimum when the two boundary components of~$A_j$ coincide, in which case~$\ell_j=A_j$ and~$\area(A_j)=0$.
In this case, $s$~measures the distance from the geodesic to the edge of the collar containing~$A_j$. Therefore, by the Collar Lemma, $\sin(\ell_j/2)=\csch(s)$, hence
\begin{align*}
\area(U)-t\, \length(\partial U)\geq&\ 2n_U\tonde{\pi/3 - t\, \arcsinh\ttonde{\csch(s)}} 
\\
\geq&\ 2\tonde{\pi/3 - t\, \arcsinh\ttonde{\csch(s)}}
\fstop
\end{align*}

Letting the right-hand side above be larger than~$\pi/3$ we get
\begin{align*}
s\geq \arcsinh\ttonde{\csch(\pi/(6t))}\comma \qquad t>1\comma \qquad  s=\log\tonde{\coth\ttonde{\tfrac{\pi}{12}} t}\fstop
\end{align*}

\iref{i:l:BD5.1:2} Let~$\mcC$ be a short collar in~$X$. For~$p\in \enl{X_\thick}{\eps+s}\cap \mcC$, by the Collar Lemma we have that
\begin{align*}
\inj_p&\geq \arcsinh\left( e^{\dist(p,\partial\mcC)}\right)=\arcsinh\left(e^{-s}\right)=\arcsinh\left(\tfrac 1 {c_1 t}\right)\geq \frac {c_2}{t}
\end{align*}
with $c_2\eqdef \arcsinh(1/c_1)$, and where the last inequality is sharp by a direct computation.

\iref{i:l:BD5.1:3}
Let~$p_1$,~$p_2\in U$.
Then we can find a rectifiable curve~$\gamma$, connecting~$p_1$ to~$p_2$, and enjoying the following properties:
\begin{enumerate}[$(a)$]
\item if~$\gamma\cap \mcC\neq \emp$, then~$\gamma\cap \partial\mcC$ consists of two points belonging to distinct connected components of~$\partial\mcC$, and~$\length(\gamma\restr_\mcC)\leq 2s$;
\item in each connected component of~$\enl{X_\thick}{\eps}$, the curve~$\gamma$ is a shortest path between its endpoints.
\end{enumerate}
See Fig.~\ref{Fig1} below.

\begin{center}\begin{figure}[htb!]
	 										\centering
	 										\def\svgwidth{350pt}
\begingroup%
  \makeatletter%
  \providecommand\color[2][]{%
    \errmessage{(Inkscape) Color is used for the text in Inkscape, but the package 'color.sty' is not loaded}%
    \renewcommand\color[2][]{}%
  }%
  \providecommand\transparent[1]{%
    \errmessage{(Inkscape) Transparency is used (non-zero) for the text in Inkscape, but the package 'transparent.sty' is not loaded}%
    \renewcommand\transparent[1]{}%
  }%
  \providecommand\rotatebox[2]{#2}%
  \newcommand*\fsize{\dimexpr\f@size pt\relax}%
  \newcommand*\lineheight[1]{\fontsize{\fsize}{#1\fsize}\selectfont}%
  \ifx\svgwidth\undefined%
    \setlength{\unitlength}{282.83722164bp}%
    \ifx\svgscale\undefined%
      \relax%
    \else%
      \setlength{\unitlength}{\unitlength * \real{\svgscale}}%
    \fi%
  \else%
    \setlength{\unitlength}{\svgwidth}%
  \fi%
  \global\let\svgwidth\undefined%
  \global\let\svgscale\undefined%
  \makeatother%
  \begin{picture}(1,0.52184507)%
    \lineheight{1}%
    \setlength\tabcolsep{0pt}%
    \put(0,0){\includegraphics[width=\unitlength,page=1]{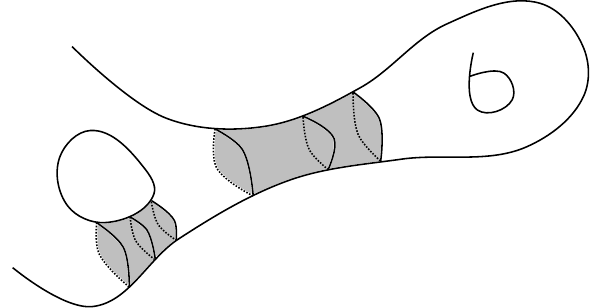}}%
    \put(0.20185476,0.16163252){\color[rgb]{0,0,0}\makebox(0,0)[lt]{\lineheight{1.25}\smash{\begin{tabular}[t]{l}$\gamma_2$\end{tabular}}}}%
    \put(0.50718487,0.33405084){\color[rgb]{0,0,0}\makebox(0,0)[lt]{\lineheight{1.25}\smash{\begin{tabular}[t]{l}$\gamma_1$\end{tabular}}}}%
    \put(0.2706125,0.0694875){\color[rgb]{0,0,0}\makebox(0,0)[lt]{\lineheight{1.25}\smash{\begin{tabular}[t]{l}$\mathcal C$\end{tabular}}}}%
    \put(0.55051817,0.20001816){\color[rgb]{0,0,0}\makebox(0,0)[lt]{\lineheight{1.25}\smash{\begin{tabular}[t]{l}$\mathcal C$\end{tabular}}}}%
    \put(0,0){\includegraphics[width=\unitlength,page=2]{figure1.pdf}}%
    \put(0.90012132,0.48770669){\color[rgb]{0,0,0}\makebox(0,0)[lt]{\lineheight{1.25}\smash{\begin{tabular}[t]{l}$p_2$\end{tabular}}}}%
    \put(-0.00185357,0.39043239){\color[rgb]{0,0,0}\makebox(0,0)[lt]{\lineheight{1.25}\smash{\begin{tabular}[t]{l}$p_1$\end{tabular}}}}%
    \put(0,0){\includegraphics[width=\unitlength,page=3]{figure1.pdf}}%
  \end{picture}%
\endgroup%

											\caption{The piecewise geodesic curve~$\gamma$ connecting~$p_1$ to~$p_2$ and~$\mathcal C$, shaded, are collars around short geodesics. }\label{Fig1}
\end{figure}\end{center}

We can decompose~$\gamma$ into its components in
\begin{align*}
X_1\eqdef \enl{X_\thick}{\eps} \qquad \text{and} \qquad X_2\eqdef \overline{\enl{X_\thick}{\eps+s}\setminus \enl{X_\thick}{\eps}}\subset \enl{X_\thin}{\eps}\fstop
\end{align*}
By the Bounded Diameter Lemma \cite{Th1978}, the length of each component of~$\gamma$ in~$X_1$ is bounded by~$4\abs\chi/\epsilon$, and we have at most~$\abs\chi$ such components.
In each connected component of~$X_2\subset \enl{X_\thin}{\eps}$ the length of~$\gamma$ is at most~$2s$, and there are at most~$\kappa$ such components.
Thus, for $s=\log(c_1t)$ we get
\begin{equation*}
\ell(\gamma)\leq \frac{4\abs\chi ^2}\epsilon+2s\kappa=\frac{4\abs\chi ^2}\epsilon+2\kappa \log(c_1t) \fstop \qedhere
\end{equation*}

\end{proof}
\end{lem}

\blue{We now show how to estimate the quantities related to~$\enl{K}{r}$ in Equation~\eqref{keyestimate}.
Let $Z$ be the zeroes of a given quadratic differential $\psi$.
}

\begin{lem}\label{areaest}
Let~$U\in \pi_0\ttonde{\enl{X_\thick}{s}}$ and~$K\eqdef \overline{X\setminus U}\cup Z$. Then, for~$r\in (0,1)$, $t>1$, and~$s=\log_+(c_1 t)$
\begin{align*}
\area\ttonde{X\setminus \enl{K}{r}}-t\, \length\ttonde{\partial\enl{K}{r}}\geq&\ \pi/3 -\kappa rt\quadre{c_7- s+ 4\pi \ttonde{1+\sinh(1)} \frac{\abs{\chi}}{\kappa}} 
\\
\geq&\ \pi/3 -\kappa r t\quadre{4\pi \ttonde{1+\sinh(1)}+c_7 }
\\
=&\ \pi/3 -\kappa r t (c_5+c_7) 
\fstop
\end{align*}

\begin{proof}
Since~$\card{Z}\leq 2\abs{\chi}$ and~$r<1$, we have that
\begin{align}
\nonumber
\length\ttonde{\partial\enl{K}{r}}\leq&\ \length(\partial U)+\length\ttonde{\partial \enl{Z}{r}} \leq \length(\partial U) + 2\pi\card{Z} \sinh(r)
\\
\label{eq:l:Extra:1}
\leq&\ \length(\partial U) + 4\pi\abs{\chi} \sinh(1) \, r \fstop
\end{align}

Furthermore,
\begin{align*}
\area\ttonde{X\setminus \enl{K}{r}}=&\ \area(X)-\area\ttonde{\enl{K}{r}}
\\
\geq&\ \area(X)-\tonde{\area\ttonde{\overline{X\setminus U}}+\area\ttonde{\enl{\partial U^+}{r}}+\area\ttonde{\enl{Z}{r}}}
\\
\geq&\ \area(U)-\area\ttonde{\enl{\partial U^+}{r}}-4\pi\abs{\chi}\ttonde{\cosh(r)-1}
\\
\geq&\ \area(U)-\area\ttonde{\enl{\partial U^+}{r}} -4\pi\abs{\chi}r
\\
\geq &\ \area(U) - t\, \area\ttonde{\enl{\partial U^+}{r}} -4\pi\abs{\chi} t r
\end{align*}
since~$t>1$.
We can estimate~$\area\ttonde{\enl{\partial U^+}{r}}$ by assuming that~$\enl{\partial U^+}{r}$ is isometrically embedded, so that, by Lemma~\ref{l:Buser},
\begin{align*}
\area\ttonde{\enl{\partial U^+}{r}}=\sum_j\ell(\gamma_j)\ttonde{\sinh(w_j-s+r)-\sinh(w_j-s)}\fstop
\end{align*}

Repeat the construction of the annuli~$A_j$ in Lemma~\ref{l:BD5.1}, and let~$w_j$ be defined as in~\eqref{eq:wAnnuli}. 
By Taylor expansion of~$\sinh$ around~$w_j-s>0$, \blue{we have that}
\begin{align*}
\area(U) - t\, \area\ttonde{\enl{\partial U^+}{r}}\geq&\ \area(U)-r t\sum_j\ell(\gamma_j)(w_j-s)\\
\geq&\ \area(U)- r t \sum_j\, \ell(\gamma_j)\, \arcsinh\ttonde{\csch(\ell(\gamma_j)/2)}
\\
&+r t\,\log(c_1 t)\sum_j \ell(\gamma_j)
\\
\geq&\ \area(U)- c_7 \kappa rt+r t\, \log(c_1 t) \sum_j \ell(\gamma_j)\fstop
\end{align*}
As a function of the metric, the summation $\sum_j\ell(\gamma_j)$ attains its maximum over the \blue{moduli space~$\mathcal M(S_{g,n})$} when~$\ell(\gamma_j)=\eps_0$ for each~$j$, thus its maximum is~$\kappa\epsilon_0$. Therefore,
\begin{align}\label{eq:l:Extra:2}
\area\ttonde{X\setminus \enl{K}{r}}\geq \area(U)-r t \kappa\, c_7+rt\kappa \epsilon_0\, \log(c_1 t)-4\pi\abs{\chi} t r
\end{align}

Multiplying~\eqref{eq:l:Extra:1} by~$-t$ and adding~\eqref{eq:l:Extra:2}, together with Lemma~\ref{l:BD5.1}\iref{i:l:BD5.1:1},  yields the conclusion.
\end{proof}
\end{lem}

Let~$U$ be the component of~$\enl{X_\thick}{\eps+s}$ containing~$p_{\max}(s)$, where~$s=\log(c_1 t)$ and~$p_{\max}$ satisfies Lemma~\ref{4.6}.
Set~$K'\eqdef \overline{X\setminus U}$ and let~$K\eqdef K'\cup Z$. \blue{This is a slight refinement of the previous $K$, in which we chose a specific component $U$ and a slightly larger neighbourhood of $U$. The next Lemma will deal with paths in $X\setminus \enl K r$. When, $r=0$ $X\setminus K=U\setminus Z$ which looks as Fig.~\ref{Fig3} below.}

\begin{center}\begin{figure}[htb!]
\centering
\def\svgwidth{250pt}
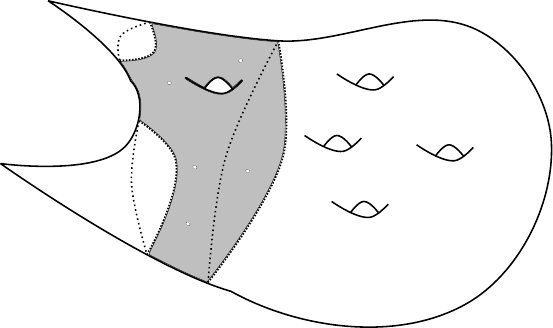
\caption{The set $X\setminus K=\text{int}(U)\setminus Z$ is greyed out and the white points are zeroes of the quadratic differential.}\label{Fig3}
\end{figure}\end{center}

\begin{lem}\label{5.2}
Fix~$t>1$. If~$r< c_2 /(\abs{\chi}t)$, then any two points in $X\setminus \enl{K}{r}$ can be joined by a rectifiable curve in $X\setminus \enl{K}{r/2}$.

\begin{proof}
\blue{We start with the following claim.}

\paragraph{Claim}
Let~$V \in \pi_0\ttonde{\enl{K}{r/2}}$. If~$V\cap\enl{K'}{r/2}\neq \emp$, then~$V\subset \enl{K'}{r}$. 

Indeed, for~$c>0$ to be fixed later, let~$V\in \pi_0\ttonde{\enl{K}{c r}}$ with~$V\cap\enl{K'}{c r}\neq \emp$. We need to show that if $V$ is such  component it does not separate $X\setminus \enl K r $. Fix~$p\in V\setminus \enl{K'}{c r}$.
Since~$V$ is connected and contained in $\enl {K}{cr}$, then~$p$ is joined to~$\enl{K'}{c r}$ by a chain of disks of radius~$c r$ centered at points in~$Z$. 
Therefore~$\dist(p, \enl{K'}{cr})\leq 2 c \card{Z} r$. 
Choosing~$c<(2\card{Z}+1)^{-1}$, e.g.~ $c\eqdef \tfrac{1}{2}(2\card{Z}+1)^{-1}$,  proves that $\dist(p, \enl{K'}{cr})\leq r/2$ and so that:
\begin{align*} \dist(p, K')&\leq \dist(p, \enl{K'}{cr})+cr=2c\card Zr +cr\leq (2\card Z+1)cr\leq r/2\comma
\end{align*}
proving that~$V\subset \enl{K'}{r}$. \blue{This concludes the proof of the claim.}

\medskip

Thus, we need to show that for $r< c_2/t$ and for all $p_0,p_1\in X\setminus \enl{K}{r}\subset \enl{U}{s}$ there exists a rectifiable curve~$\gamma\subset X\setminus \enl{K}{c_2 r}$ connecting~$p_0$ to~$p_1$. By the Collar Lemma,
\begin{equation*}
\inj_{\enl{U}{s}}\eqdef \min_{p\in \enl{U}{s}} \inj_p\geq \arcsinh(e^{-s})=\arcsinh\tonde{\frac{1}{c_1t}}\geq \frac{c_2}{t}\comma
\end{equation*}
similarly to the proof of Lemma~\ref{l:BD5.1}\iref{i:l:BD5.1:2}.

Now, argue by contradiction \blue{and assume} that there exists no rectifiable curve as in the assertion.
Then, there exists a rectifiable loop~$\alpha$ in~$\enl{Z}{r/2}$ separating~$X\setminus \enl{K}{r}\subset \enl{U}{s}$ into connected components so that~$p_0$ and~$p_1$ belong to two distinct such components. See the picture in Fig.~\ref{Fig2} below.
\begin{center}\begin{figure}[htb!]
\centering
\def\svgwidth{350pt}
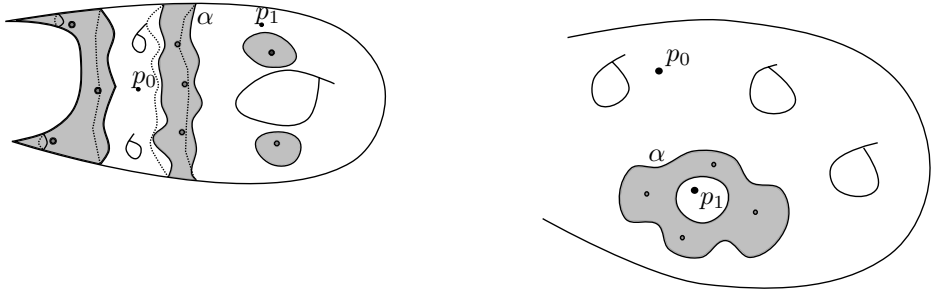
\caption{The two cases for the loop~$\alpha$ separating~$p_1$ to~$p_2$. The shaded regions are part of $\enl{K}{r/2}$ and the grey dots are zeroes of the quadratic differential.}\label{Fig2}
\end{figure}\end{center}

For any such~$\alpha$,
\begin{equation*}
\length(\alpha)\leq r \card{Z}<\card{Z}\frac {c_2}{\abs \chi t}\leq \frac {c_2} t\leq \inj_{\enl{U}{s}}\fstop
\end{equation*}
As a consequence,~$\alpha\subset \enl{U}{s}$ is null-homotopic and so we must be \blue{on the right side} of Fig.~\ref{Fig2}. Therefore, there exists $L\in \R^+$ such that $\alpha\subset B_L(q)$ for $q\in (U)_s$ and $L\leq \ell(\alpha)/2< r/2$. Thus, the component $W\subset X\setminus \enl{K}{r}$ containing, say, $p_1$, lies in $B_L(q)\subset B_r(q)$ and note that by construction its distance from any zero is at least~$r$.
Therefore, $W$ is at distance $r/2+L< r$ from a zero. However, since $d(W,Z)\geq r$ we have a contradiction.
\end{proof}
\end{lem}

We now state the main Lemma we will use in our estimate of \eqref{keyestimate}.

\begin{lem}\label{5.3}
Let~$r<c_2/(\abs\chi t)$, and set~$a_1\eqdef 4\abs{\chi}^2/\epsilon +2\kappa \log c_1+2\, c_2\, c_3$.
Then, any two points in $\overline{X\setminus \enl{K}{r}}$ are joined by a rectifiable curve~$\gamma\subset \overline{X\setminus \enl{K}{r/2}}$ with the following properties:
\begin{enumerate}[$(i)$]
\item\label{i:l:BD5.3:1} $\gamma$ consists of length-minimising geodesic segments and at most one arc in each of the components of $\partial \enl{K}{r/2}$;
\item\label{i:l:BD5.3:2} $\ell(\gamma)\leq a_1+ 2\kappa \log t$;
\item\label{i:l:BD5.3:3} for $z\in Z$: $\length\ttonde{\gamma\cap B_w(z)}\leq 2(1+c_3) w$ for all $w>0$ such that~$B_w(z)$ is embedded.
\end{enumerate}

\begin{proof}
\iref{i:l:BD5.3:1}--\iref{i:l:BD5.3:2} Fix points~$p_0,p_1\in \overline{X\setminus \enl{K}{r}}$.
By Lemma \ref{l:BD5.1} there exists a rectifiable~$\gamma\subset U$ connecting them, with
\begin{equation*}
\ell(\gamma)\leq \frac{4\abs{\chi}^2} \epsilon +2\kappa \log c_1 +2\kappa \log t \fstop
\end{equation*}

The curve~$\gamma$ intersects~$\enl{K}{r/2}$ in at most~$2\abs\chi$ components (i.e.\ balls around zeroes of~$\psi$). In each such component $V= B_{r/2}(z)$ (for some~$z\in Z$) we can replace~$\gamma\restr_V$ by a shortest path on~$\partial V$ as the one in Lemma \ref{l:BD5.1} \ref{i:l:BD5.1:3}.

Since~$V$ is a ball, the length of~$\gamma\restr_V$ is bounded by half the length of the circumference of a great circle on~$V$, i.e.\begin{equation}\label{eq:l:BD5.3:1}
\pi\sinh(r/2)\leq \pi\sinh(r)\leq c_3 \, r \comma \qquad r<\frac{c_2}{\abs{\chi} t}<\frac{c_2}{\abs{\chi}}\fstop 
\end{equation}
 
By repeating this reasoning on each component~$V$ as above, we obtain a path $\gamma'\colon p_0\rar p_1$ satisfying~\iref{i:l:BD5.3:1} and such that:
\begin{align*}
\length(\gamma')&\leq\ell(\gamma)+2\abs\chi c_3\,r
\\
&\leq \frac{4\abs{\chi}^2} \epsilon +2\kappa\log c_1 +2\kappa \log t+2 \frac {c_2 c_3}{t} && (t>1)
\\
&\leq  \frac{4\abs{\chi}^2} \epsilon + 2\kappa\log c_1 +2\, c_2\, c_3+2\kappa \log t\fstop
\end{align*}

\iref{i:l:BD5.3:3} Let $z\in Z$ be a zero of~$\psi$ and fix~$w>0$.
Each component~$\alpha$ of~$\gamma$ in~$\partial \enl{K}{c_2r}$ has length at most $c_3\, r$ and each geodesic arc of~$\gamma$ connecting an endpoint of~$\alpha$ to~$\partial B_w(z)$ has length at most~$w$.
We now estimate
\begin{equation*}
\card{\pi_0\ttonde{\gamma\cap \overline{B_w(z)}}}\leq
\begin{cases}
0 & w<r/2
\\ 
1 & p_1,p_2\in \overline{B_w(z)}
\\ 
1 & p_1\in  \overline{B_w(z)}\comma p_2\not\in \overline{B_w(z)}
\\
1 & p_1,p_2\not\in \overline{B_w(z)}
\end{cases}
\end{equation*}

The first bound holds by definition.
The second holds by the convexity of hyperbolic balls: if $p_1,p_2\in B_w(z)$ then we can choose~$\gamma\subset B_w(z)$.
The third and fourth one follow from the fact that if~$\gamma$ has more than one component in~$B_w(z)$, then we can shortcut~$\gamma$ inside the ball.

If~$w=r/2$, then~$\gamma\restr_{\overline{B_w(z)}}\subset \partial B_w(z)$, and we may choose~$\gamma\restr_{B_w(z)}$ to be a circumference arc, so that~$\length\ttonde{\gamma\restr_{B_{r/2}(z)}}\leq \pi\sinh (r/2)\leq c_3 r$ by~\eqref{eq:l:BD5.3:1}.

If instead~$w>a_1 r$, then we may choose~$\gamma$ to be either a geodesic segment, or a union~$\gamma_1\cup\gamma_{2}\cup\gamma_3$, where~$\gamma_1$ and $\gamma_{2}$ are geodesic segments each connecting~$\partial B_w(z)$ to~$\partial B_{r/2}(z)$, and~$\gamma_3$ is a circumference arc on~$\partial B_{r/2}(z)$.
In the first case,~$\length(\gamma\restr_{B_w(z)})\leq 2w$.
In the second case,
\[
\length\ttonde{\gamma\restr_{B_w(z)}}\leq 2w+\pi\sinh(a_1 r)\leq 2w+c_3 r\leq 2w+c_3 r\fstop
\]
Thus, we obtain that:
\begin{equation*}
\length\ttonde{\gamma\cap \overline{B_w(z)}}\leq
\begin{cases}
0 & \text{if } w< r/2
\\
\displaystyle 2c_3 w & \text{if } w=r/2
\\
2w+c_3w & \text{if } w\geq r/2
\end{cases}\qquad\qquad \leq 2(1+c_3)w \comma
\end{equation*}
which concludes the proof.
\end{proof}
\end{lem}

\blue{With~$m(r)$ as in~\eqref{eq:DefMr} we can now estimate~\eqref{keyestimate} and show our final result.}

\begin{proof}[Proof of Theorem~\ref{t:Main}]
Let~$r<c_2/\abs{\chi}\leq 2$.
Let $s$ be as in Lemma \ref{areaest} and choose $U$ to be the component of $X_\thick(s)$ containing the point $p_{\max}(s)$ as in Lemma \ref{4.6}.
Let $Z$ be the set of zeroes of $\psi$, $K\eqdef \overline{X\setminus U} \cup Z$, and $K'\eqdef  \overline{X\setminus U}$.
Let $W\eqdef \enl{X_\thick}{s+1}\setminus Z$, $p_1\in\partial \enl{K}{r}$, and $p_2=p_{\max}(s)\in \enl{X_\thick}{s}\setminus Z$.
Therefore, we have that $\langle \psi(p_2)\rangle=\UBL(s)$.
Moreover, let $\gamma\subset W$ be a path from $p_1$ to $p_2$ satisfying the conditions of Lemma \ref{5.3} and note that
\begin{equation*}
\dist(p,\partial W)\geq \min \set{1,\dist(p,Z)}\comma \qquad p\in\gamma \fstop
\end{equation*}

By Lemma \ref{4.6}\iref{i:l:BD4.6:2} we have that:
\[
\UBL(s+1)\leq e\cdot \UBL(s) \fstop
\]

By Theorem \cite[4.4]{BarDil96}, we have that:
\begin{align*}
\langle \psi(p_1)\rangle&\geq \langle \psi(p_2)\rangle \left(\frac{ \langle \psi(p_2)\rangle}{c_4 \cdot\UBL(s+1)}\right)^{-1+\exp\tonde{ \int_\gamma\frac{\diff s}{\tanh(1\wedge \dist(\gamma_s,Z))}}}
\\
&\geq \UBL(s)\left(\frac 1 {e\, c_4}\right)^{-1+\exp\tonde{\int_\gamma \coth\ttonde{1\wedge \dist(\gamma_s,Z)} \diff s}}
\\
&\geq e\, c_4\, \UBL(0)\cdot \tonde{\frac{1}{e\, c_4}}^{\exp\tonde{\int_\gamma \coth\ttonde{1\wedge \dist(\gamma_s,Z)} \diff s}} \comma
\intertext{where we can estimate~$\UBL(0)$ by Lemma~\ref{4.6}\ref{i:l:BD4.6:1},}
&\geq \frac{e\,c_4\,\ell}{16\abs{\chi}} \norm{\psi} \tonde{\frac{1}{e\, c_4}}^{\exp\tonde{\int_\gamma \coth\ttonde{1\wedge \dist(\gamma_s,Z)} \diff s}} 
\\
&=\frac{e\,c_4\, \ell}{16\abs{\chi}} \norm{\psi} \exp\tonde{-\log(e\,c_4)\exp\tonde{\int_\gamma\coth\ttonde{1\wedge \dist(\gamma_s,Z)} \diff s }}
\end{align*}

We now estimate~$\int_\gamma \coth(1\wedge  \dist(\gamma_s,Z)) \diff s$ from above by breaking it into two terms:
\begin{equation*}
\int_\gamma\frac{\diff s}{\tanh\ttonde{1\wedge  \dist(\gamma_s,Z)}}\leq \int_{\gamma\setminus Z(1)}\diff s+\int_{\gamma\cap Z(1)}\frac{\diff s}{\dist(\gamma_s,Z)}
\end{equation*}
The first term is bounded by $\ell(\gamma)$ while for the second term we have by Lemma~\ref{5.3}\iref{i:l:BD5.3:1}
\begin{align*}
\int_{\gamma\cap Z(1)}\frac{\diff s}{\dist(\gamma_s,Z)}\leq& \int_1^{\frac{2}{r}} \length\ttonde{\gamma\cap \enl{Z}{1/u}} \diff u
\\
\leq& \int_1^{\frac{2}{r}}\frac{2(1+c_3)}{u^2} \diff u = 2(1+c_3)(1-r/2)
\end{align*}
since~$r\leq 2$.

By Lemma \ref{5.3}\iref{i:l:BD5.3:1} we have that:
\begin{equation*}
\ell(\gamma)\leq a_1+2\kappa\,\log t
\end{equation*}

Thus:
\begin{align*}
\int_\gamma\frac{\diff s}{\tanh\ttonde{1\wedge  \dist(\gamma_s,Z)}}\leq &\ a_1+2\kappa\,\log t+2(1+c_3)(1-r/2) 
\\
=&\ a_1+2(1+c_3)+2\kappa\,\log t-(1+c_3) r
\end{align*}

Therefore, since~$\log(e\,c_4)>0$, for all $p_1\in\partial \enl{K}{r}$ we get:
\be
\langle \psi(p_1)\rangle\geq \frac{e\, c_4\, \ell}{16\abs\chi} \norm{\psi} \exp\tonde{-\log(e\, c_4) \exp\ttonde{ a_1+2(1+c_3)+2\kappa\,\log t-(1+c_3) r} }
\ee
Thus, by minimizing over~$p_1\in\partial \enl{K}{r}$ we obtain:
\be
m(r) \geq \frac{e\, c_4\,\ell}{16\abs\chi}\norm{\psi} \exp\tonde{-\log(e\, c_4) \exp\ttonde{ a_1+2(1+c_3)+2\kappa\,\log t-(1+c_3) r} }
\ee
which for $a_2\eqdef \log(e\, c_4) \, e^{a_1+2(1+c_3)}>0$ can be rewritten as:
\be 
m(r)\geq e\, c_4\, \frac{\ell}{16\abs\chi}\norm{\psi} \exp\ttonde{-a_2 t^{2\kappa} e^{-(1+c_3) r} }
\ee

Then, equation~\eqref{keyestimate} with $K\eqdef W$ becomes, for $r_0<\frac 1{4t}$,

\begin{align*}
1-\norm\psi\geq&\ \int_0^{r_0} m(r) \tonde{ t^{-1} \area\ttonde{X\setminus \enl{K}{r}}-\length\ttonde{\partial \enl{K}{r}} }\diff r\fstop
\intertext{By Lemma~\ref{areaest} we thus have that, for every~$r_0<\tfrac{1}{4t}$,}
1-\norm\psi \geq&\ \frac{e\, c_4\, \ell}{16\abs\chi t}\norm{\psi} \int_0^{r_0}  \exp\tonde{-a_2 t^{2\kappa} e^{-(1+c_3) r}} \ttonde{\pi/3 -\kappa r t (c_5+c_7) } \diff r
\\
\geq&\ \frac{e\, c_4\, \ell\, e^{-a_2 t^{2\kappa}}}{16\abs\chi t}\norm{\psi} \int_0^{r_0} \ttonde{\pi/3 -\kappa r t (c_5+c_7) } \diff r \fstop
\intertext{Maximizing over~$r_0\in \ttonde{0, \tfrac{1}{4t}}$ additionally so that the integrand is non-negative, we have therefore that} 
1-\norm\psi \geq &\ \frac{e\, c_4\,\ell\, e^{-a_2 t^{2\kappa}}}{16\abs\chi t}\norm{\psi} \int_0^{\frac{1}{4t}\wedge \frac{\pi}{3\kappa t (c_5+c_7)}} \ttonde{\pi/3 -\kappa r t (c_5+c_7) } \diff r
\\
=& \ \frac{e\, \pi^2\, c_4}{288\, \kappa (c_5+c_7)} \frac{\ell\, e^{-a_2 t^{2\kappa}}}{\abs\chi t^2}\norm{\psi}\comma
\end{align*}
and maximizing the right-hand side over~$t>1$, i.e.\ choosing $t=1$, we conclude that
\begin{align*}
\norm{\psi}\leq& \frac{1}{1+\displaystyle \frac{C\,\ell\, e^{-a_2}}{\kappa\abs{\chi}} }\comma \qquad C\eqdef \frac{e\, \pi^2\, c_4}{288\, (c_5+c_7)} \fstop \qedhere
\end{align*}
\end{proof}

\paragraph{Contraction factors of skinning maps}
We now apply our explicit bounds from Theorem \ref{t:Main} to get effective bounds on the contraction factor of the skinning map.

Let $N\in AH(M,\mathcal P)$ be a pared acylindrical manifold so that
\begin{itemize}
\item $\mathcal P\subset\partial M$ is a collection of pairwise disjoint closed annuli and tori; 
\item $\mathcal P$ contains all tori components of $M$ and $M$ is acylindrical relative to $\mathcal P$.
\end{itemize}
Let $\partial_0 M\eqdef \partial M\setminus\mathcal P$.
By \cite[p.~443]{McM1989} we have that, for every such~$N$,
 \begin{equation*}
\abs{\diff \sigma}=\abs{\diff \sigma^*} \fstop
\end{equation*}
By Theorem \ref{t:Main},
 \begin{equation*}
 \diff\sigma^*(\phi)=\sum_{U\in BN} \Theta_{U/X} \left(\phi\vert_U\right)\leq \max_{X\in \partial_0M} \frac{1}{1+C_{g,n,\ell}} \norm\phi\comma 
 \end{equation*}
where $\ell$ is the injectivity radius of the conformal boundary~$\partial_\infty N$, and $C_{g,n,\ell}$. Thus, we obtain the following corollary:

\begin{cor}
Let~$(M,\mathcal P)$ be a pared acylindrical hyperbolic manifold.
Then, the skinning map at $N\in AH(M,\mathcal P)$ has contraction factor bounded by
\be \abs{d\sigma}\leq  \max_{X\in \partial_0M} \frac{1}{1+C_{g,n,\ell}} \fstop\ee
\end{cor}

			\medskip
	\small
	
%
%
%
%
%
%
%
%
	\end{document}